\documentclass[10pt]{amsart}

\usepackage[latin1]{inputenc}
\usepackage[english]{babel}
\usepackage{amsmath}
\usepackage{amsfonts}
\usepackage{amssymb}
\usepackage{amsthm}
\usepackage{epsfig}
\usepackage{graphicx,colortbl}
\usepackage{url,hyperref}
\usepackage{mathtools, color}
\usepackage{dsfont}

\newtheorem{theorem}{Theorem}[section]
\newtheorem{lemma}{Lemma}[section]

\newcounter{theor}

\def\R{\mathbb{R}}
\def\N{\mathbb{N}}
\def\Z{\mathbb{Z}}

\def\vol{\mathrm{Vol}}

\def\e{\mathrm{e}}

\def\E{\mathbb{E}}

\def\Var{\text{Var}}
\def\var{\text{Var}}

\def\dist{\text{dist}}

\def\e{\mathrm{e}}

\def\2Z{2^{-m}\Z^n}

\numberwithin{equation}{section}

\begin{document}
\title[Limit theorems for the distance between random points in $l_p^n$-balls]{Limit theorems for the distance between two random points in $l_p^n$-balls}
\author[D.\,Alonso]{David Alonso-Guti\'errez}
\address{\'Area de an\'alisis matem\'atico, Departamento de matem\'aticas, Facultad de Ciencias, Universidad de Zaragoza, Pedro Cerbuna 12, 50009 Zaragoza (Spain), IUMA}
\email[(David Alonso)]{alonsod@unizar.es}
\author[J. Mart\'in]{Javier Mart\'in Go\~ni}
\address{University of Passau. Faculty of Computer Sciences and Mathematics. Innstrasse 41, 94032 Passau, Germany.}
\email{j.martin@unizar.es, javier.martingoni@uni-passau.de}
\author[J. Prochno]{Joscha Prochno}
\address{University of Passau. Faculty of Computer Sciences and Mathematics. Innstrasse 41, 94032 Passau, Germany.}
\email{Joscha.Prochno@uni-passau.de}
\thanks{The first and second named author are supported by Project PID2022-137294NB-I00 and PID2025-169474NB-C22 funded by  MICIU/AEI/10.13039/501100011033/ and by DGA Project E48\_26R. The second and third authors are supported by the DFG project 516672205.}
\begin{abstract}
In this paper, we prove that the Euclidean distance between two independent random vectors uniformly distributed on $l_p^n$-balls ($1\leq p\leq\infty$) or on their boundaries according to the cone measure satisfies a central limit theorem as $n\to\infty$. Also, we give a compact proof of the case of the sphere, which was proved by Hammersley. Furthermore, we complement our central limit theorems by providing large deviation principles for points in the balls and on their boundaries when $1\leq p\leq\infty$. The speed is $n$ for $p\geq2$, whereas the regime
$1\leq p<2$ has speed $n^{p/2}$ and an explicit rate function.
\end{abstract}
\maketitle

\section{Introduction, notation and results.}
\label{section1}

Recent developments in Asymptotic Geometric Analysis have significantly deepened our understanding of high-dimensional spaces. This interdisciplinary field, which merges techniques from analysis, geometry, and probability, has unveiled striking regularities and patterns in geometric structures as the dimension grows. In particular, the behavior of random objects in high dimensions has been a central focus, revealing connections to classical probabilistic phenomena.

One prominent area of study concerns central limit theorems (CLT) in high-dimensional geometry. Early works by Poincar\'e and Borel (see \cite{DF} or \cite{St} for modern expositions) demonstrated that as the dimension of Euclidean spaces increases, the distribution of random points in high-dimensional spheres or balls tends to converge to a Gaussian distribution. This convergence has since been extended to more general convex bodies, notably through Klartag's CLT \cite{K,K2}, which establishes that projections of random vectors on high-dimensional convex sets exhibit Gaussian-like behavior.

Further advances in the field have focused on understanding the probabilistic structure of high-dimensional bodies, such as cubes and  $l_p^n$-balls. The work of Paouris, Pivovarov, and Zinn \cite{PPZ} has contributed important results in this direction, providing a central limit theorem for the volume of $k$-dimensional random projections ($k$ fixed) of the $n$-dimensional cube as the space dimension $n$ tends to infinity. Prochno, Th\"ale and Tuchel \cite{PTT} generalized and complemented this considerably by establishing a central limit theorem and moderate and large deviation principles for the volume of random projections and sections of $l_p^n$-balls.

Beyond these normal fluctuations, there has been growing interest in understanding large deviations in high-dimensional spaces. These deviations capture rare events where the behavior of random vectors significantly differs from typical Gaussian outcomes, highlighting non-universal features of different convex bodies. In particular, the study of large deviation principles (LDPs) has provided critical insights into these rare events, offering a more complete picture of the probabilistic behavior of random vectors in convex bodies; we refer the interested reader to \cite{P} and the references cited therein.

The aim of this paper is to investigate CLTs and LDPs related to the Euclidean distance between random vectors in $l_p^n$-balls, to uncover deeper connections between geometric properties and probabilistic phenomena in high-dimensional settings. This problem was first researched by Hammersley in \cite{H}, where he proved an initial CLT for the case of the hypersphere. Further, it was extended to spherical distributions by Lord in \cite{L}. In this paper, we extend those results to $l_p^n$-balls and their respective boundaries. Furthermore, we obtain LDPs for $1\leq p\leq\infty$, with a change of speed at $p=2$.

We start with the central limit theorems and then present the corresponding large deviations counterparts. For $p\in[1,\infty]$ and $x=(x_1,\ldots,x_n)\in\R^n$, $n\in\N$, we recall the definition of the $\Vert\cdot\Vert_p$-norm as
\begin{align*}
\Vert x\Vert_p:=
\begin{cases}
\left(\sum_{i=1}^n|x_i|^p\right)^{1/p},&p<\infty,\\
\max_{1\leq i\leq n}|x_i|,&p=\infty.
\end{cases}
\end{align*}
We denote the corresponding unit ball by
\begin{align*}
B_p^n:=\{x\in\R^n:\Vert x\Vert_p\leq1\},
\end{align*}
and write $\partial B_p^n$ for its boundary. We use $X\sim\text{Unif}(B_p^n)$ for the uniform distribution with respect to volume, and use $X\sim\text{Unif}(\partial B_p^n)$ for the cone probability measure introduced in Section~2.2. Throughout the paper, the parameter $p$ is fixed as $n\to\infty$. For two random vectors $X^{(n)},Y^{(n)}\in\R^n$, we write
\begin{align}
D_{n,p}:=n^{1/p-1/2}\Vert X^{(n)}-Y^{(n)}\Vert_2,
\label{EqNormalizedDistance}
\end{align}
with the convention $1/\infty=0$. Thus, $D_{n,p}$ is a normalised Euclidean distance, which agrees with the usual distance when $p=2$. We will use the notations $D^\partial_{n,p}$ (resp. $D^B_{n,p}$) whenever we want to stress out the fact that the random vectors are distributed on the boundary of $B_p^n$ (resp. in $B_p^n$).

\subsection{Central limit theorems}

A cornerstone in probability theory is the central limit theorem. A sequence $(X_n)_{n \in \N}$
of random variables is said to satisfy the central limit theorem if
\begin{align*}
\frac{X_n - \mathbb{E} X_n}{\sqrt{\Var X_n}} \overset{d}{\rightarrow} N(0,1),
\end{align*}
where $\overset{d}{\rightarrow}$ denotes convergence in distribution as $n\to\infty$. That is, for every $t \in \R$,
\begin{align*}
\lim _{n \to \infty} \mathbb{P} \left( \frac{X_n - \mathbb{E} X_n}{\sqrt{\Var X_n}} \leq t \right) - \Phi (t) = 0,
\end{align*}
where $\Phi (t)$ is the distribution function of the standard normal distribution.
More generally, we say that a sequence of random variables $(Y_n)_{n \in \N}$ satisfies a CLT if $Y_n \overset{d}{\to} N(0, \sigma ^2)$, for some $\sigma > 0$.

\medskip

This first theorem establishes a central limit theorem for the distance between $X^{(n)}$ and $Y^{(n)}$, denoted $D_{n,p}$, which are random vectors uniformly distributed in $l_p^n$-balls or on their boundary. For $1\leq p<\infty$ and $\alpha\geq0$, let
\begin{align*}
M_p(\alpha):=\frac{\Gamma((\alpha+1)/p)}{\Gamma(1/p)},
\qquad
c_p:=\sqrt{2\,p^{2/p}M_p(2)},
\end{align*}
where $\Gamma$ is the gamma function.

\begin{theorem}
\label{Thm:CLTforlpballs}
Let $1\leq p<\infty$, and let $X^{(n)},Y^{(n)}$ be independent random vectors uniformly distributed in $B_p^n$, or on $\partial B_p^n$ according to the cone measure. Then,
\begin{align*}
\sqrt n\left(D_{n,p}-c_p\right)
\overset d\longrightarrow N(0,\sigma_p^2),
\end{align*}
where
\begin{align}
\sigma_p^2:=\frac{p^{2/p}}{4M_p(2)}
\left[M_p(4)+\left(1-\frac4p\right)M_p(2)^2\right]>0.
\label{EqCLTVariance}
\end{align}
\end{theorem}

\medskip

Note that the constants in Theorem~\ref{Thm:CLTforlpballs} satisfy
\[
\lim_{p\to\infty}c_p=\sqrt{\frac23},
\qquad
\lim_{p\to\infty}\sigma_p^2=\frac7{30}.
\]
These constants coincide with the constants in the following central
limit theorem for the cube and its boundary, which we
prove separately.

\begin{theorem}
\label{Thm:CLTforCube}
Let $X^{(n)},Y^{(n)}$ be independent random vectors uniformly distributed in $B_\infty^n$, or on $\partial B_\infty^n$ according to the cone measure. Then,
\begin{align*}
\sqrt n\left(D_{n,\infty}-\sqrt{\frac23}\right)
\overset d\longrightarrow N\left(0,\frac7{30}\right).
\end{align*}
\end{theorem}

\subsection{Large deviation principles}

Large deviation principles describe the asymptotic behavior of the probability of rare events for a sequence of random variables by characterizing their exponential decay rate in terms of the speed and a rate function (see section \ref{Subsection:LargeDeviations} for precise definitions). For $2\leq p<\infty$, we first establish the LDP for the
boundary by applying Cram\'er's theorem and the contraction
principle to five empirical means. The result for the
ball then follows by including the two independent radial
variables and applying the contraction principle once more.

\medskip

For $2\leq p<\infty$ and
$t=(t_1,\ldots,t_5)\in\mathbb R^5$, define
\begin{align}
\widetilde\Lambda_p(t)
&:=\log\Bigg[
\frac{1}{4\Gamma(1+1/p)^2}
\int_{\mathbb R^2}
\exp\bigl(t_1x^2+t_2y^2+t_3xy\bigr)
\notag\\
&\hspace{28mm}\times
\exp\bigl((t_4-1)|x|^p+(t_5-1)|y|^p\bigr)
\,dx\,dy
\Bigg].
\label{EqLogMGFFiveIntro}
\end{align}
The integral is allowed to be infinite. We denote the
Legendre--Fenchel transform of $\widetilde\Lambda_p$
by $\widetilde\Lambda_p^*$, and put
\[
\mathcal E
:=\bigl\{(a,b,c,u,v)\in\mathbb R^5:
a,b\geq0,\ u,v>0,\ c^2\leq ab\bigr\}.
\]

\begin{theorem}
\label{ThmLDPpgeq2}
Let $2\leq p<\infty$, and let $X^{(n)},Y^{(n)}$
be independent random vectors.

If $X^{(n)},Y^{(n)}$ are distributed on $\partial B_p^n$
according to the cone probability measure, then
$(D_{n,p})_{n\in\mathbb N}$ satisfies an LDP on $\mathbb R$
with speed $n$ and good rate function $I_p$, given by
\begin{equation}
I_p(z)
=\inf_{\substack{(a,b,c,u,v)\in\mathcal E\\
au^{-2/p}+bv^{-2/p}-2c(uv)^{-1/p}=z^2}}
\widetilde\Lambda_p^*(a,b,c,u,v),
\qquad z\geq0.
\label{EqRatePgeqTwo}
\end{equation}

If $X^{(n)},Y^{(n)}$ are uniformly distributed in $B_p^n$,
then $(D_{n,p})_{n\in\mathbb N}$ satisfies an LDP on
$\mathbb R$ with speed $n$ and good rate function $I_p^B$,
given by
\begin{equation}
I_p^B(z)
=\inf_{\substack{(a,b,c,u,v)\in\mathcal E,\ 0<r,s\leq1\\
r^2au^{-2/p}+s^2bv^{-2/p}-2rsc(uv)^{-1/p}=z^2}}
\left\{
\widetilde\Lambda_p^*(a,b,c,u,v)-\log r-\log s
\right\},
\qquad z\geq0.
\label{EqRateBallIntro}
\end{equation}
Both rate functions are infinite for $z<0$, and the
infimum over the empty set is understood to be infinite.
\end{theorem}

We now turn to the case of the $n$-dimensional cube, where we establish a large deviations principle for points uniformly distributed in the cube.
\begin{theorem}
\label{ThmLDPBoundaryCube}
Let $X^{(n)},Y^{(n)}$ be independent random vectors uniformly distributed in $B_\infty^n$, or on $\partial B_\infty^n$ according to the cone measure. Then $(D_{n,\infty})_{n\in\N}$ satisfies an LDP on $\R$ with speed $n$ and good rate function
\begin{align*}
I_\infty(z):=
\begin{cases}
\Lambda_\infty^*(z^2),&z\in[0,2],\\
+\infty,&z\notin[0,2],
\end{cases}
\end{align*}
where $\Lambda_\infty^*$ is the Legendre--Fenchel transform of
\begin{align*}
\Lambda_\infty(t):=\log\left[\int_0^2e^{tx^2}
\left(1-\frac x2\right)\,dx\right],\qquad t\in\R.
\end{align*}
\end{theorem}

Finally, when $1\leq p<2$, we use large deviation estimates for sums with stretched exponential tails. These estimates lead to a slower speed and gives the following explicit rate function.

\begin{theorem}
\label{ThmLDPpless2}
Let $1\leq p<2$, and let $X^{(n)},Y^{(n)}$ be independent
random vectors uniformly distributed in $B_p^n$, or on
$\partial B_p^n$ according to the cone measure.
Then $(D_{n,p})_{n\in\N}$ satisfies an LDP on $\R$
with speed $n^{p/2}$ and good rate function
\begin{align}
I_p(z):=
\begin{cases}
\displaystyle\frac{2^{1-p}}p(z^2-c_p^2)^{p/2},
&z\geq c_p,\\[1mm]
+\infty,&z<c_p.
\end{cases}
\label{EqRatePlessTwo}
\end{align}
\end{theorem}

In particular, $p=1$ is included, with $c_1=2$
and speed $\sqrt n$.

\medskip

The rest of this paper is structured as follows. Some notation and results needed for the main results are collected in Section \ref{section2}. In Section \ref{section3}, we introduce the problem considered and prove the spherical case following the idea of~\cite{H} in a more direct way. Section~\ref{section4} contains the proofs of the central limit theorems, Section~\ref{section5} is devoted to the large deviation principles for $2\leq p\leq\infty$, and Section~\ref{section6} treats the regime $1\leq p<2$.

 \section{Notation and preliminaries}

 \label{section2}

\subsection{General notation}

For a subset $A$ of a topological space $\mathbb{X}$, we denote its interior by $A^\circ$ and its closure by $\overline{A}$. Given a Borel probability measure $\tau$ on $\R^n$, we indicate by $X \sim \tau$ that the random vector $X$ has distribution $\tau$. In particular, we denote as $X \sim \mathcal{N}( \mu , \sigma ^2 )$ that the random variable $X$ has a Gaussian distribution with mean $\mu \in \R$ and variance $\sigma ^2 > 0$. In dimension $n \geq 2$, we write $\mathcal{N} ( \mu , \Sigma)$ to indicate the multivariate Gaussian distribution with mean $\mu \in \R^n$ and covariance matrix $\Sigma \in \R^{n \times n} $.

For sequences of random vectors in a fixed Euclidean space, we write $\overset{d}{\to}$, $\overset{p}{\to}$ and $\overset{a.s.}{\to}$ for convergence in distribution, in probability and almost surely, respectively. Moreover, given two random elements $X_1$ and $X_2$, we indicate equality in distribution as $X_1 \overset{d}{=} X_2$. 
We write $\E$ and $\var$ for the expected value and variance, respectively.

\subsection{Probabilistic aspects of \texorpdfstring{$l_p^n$}{l\string_p\string^n}-balls}

\label{Section2.2}

Recall that for $p \in [1, \infty ]$ and $x = (x_1, ... , x_n) \in \R^n$, $n \in \N$, we write $\Vert x \Vert _p$ for the $p$-norm of $x$. We denote the unit ball in $\R^n$ with respect to the $\Vert \cdot \Vert _p$-norm by $B_p^n$, that is

\begin{align*}
B_p^n := \lbrace x \in \R^n \; \Vert x \Vert _p \leq 1 \rbrace,
\end{align*}

and we write $\partial B_p^n$ for its boundary. Additionally, we will write $S^{n-1}:= \partial B_2^n$ to denote the Euclidean sphere.

Let $\sigma _p ^n$ be the surface measure (Hausdorff measure) on $\partial B_p^n$, $p \geq 1$, and denote by $\mu_p ^n$ the cone probability measure on $\partial B_p^n$, defined by

\begin{align*}
\mu_p ^n (A) = \frac{1}{\vol (B_p^n)} \vol \left( \lbrace ta \in \R^n ; a \in A, 0\leq t \leq 1 \rbrace \right), \qquad A\subseteq\partial B_p^n\text{ Borel},
\end{align*}

where $\vol$ denotes the Lebesgue measure.

We use the following normalisation of the $p$-generalized Gaussian distribution ($1\leq p<\infty$): its density function $f_p$ is given by

\begin{align*}
f_p(x) = C_p^{-1} e^{- |x|^p}, \qquad x\in\R,
\end{align*}

with respect to the Lebesgue measure on $\R$, where the normalization constant $C_p$ is given by $C_p:= 2  \Gamma \left(1 + \frac{1}{p} \right)$. For the case $p=\infty$, we define the $\infty$-generalized Gaussian as a random variable uniformly distributed on $[-1,1]$. Notice that, if $g$ is a $p$-generalized Gaussian for some $p \in [1, \infty)$, then for every $\alpha \geq 0$, $\E | g |^\alpha = M_p(\alpha )$.

We recall from \cite{SZ} the following probabilistic representation for a uniformly distributed random vector in $B_p^n$ and in $\partial B_p^n$.

\begin{lemma}[Schechtman and Zinn]

\label{Lem:SchAndZinn}

Fix $1\leq p < \infty$. Let $Y = (Y_1, ... , Y_n)$ be a random vector with independent $p$-generalized Gaussian coordinates $Y_1,...,Y_n$, and let $U$ be a random variable uniformly distributed on $[0,1]$, independent of $Y$. Then,

\begin{itemize}

\item The random vector $\displaystyle{\frac{Y}{\Vert Y \Vert _p}}$ and the random variable $\Vert Y \Vert _p$ are independent.

\item $\displaystyle{\frac{Y}{\Vert Y \Vert _p}}$ is uniformly distributed on $\partial B_p^n$ according to the cone measure $\mu _p^n$.

\item $\displaystyle{U^{1/n} \frac{Y}{\Vert Y \Vert _p}}$ is uniformly distributed in $B_p^n$.

\end{itemize}

\end{lemma}

\subsection{Delta Method}

We will use the following technical lemma, which characterizes the asymptotic distribution of a differentiable function of a random vector which is asymptotically Gaussian (see \cite{O}).

\begin{lemma}[Delta method]

\label{DeltaMethod}

Let $(X_n)_{n\in\N}$ be a sequence of $k-$dimensional random vectors such that $\sqrt{n}(X_n - \mu)$ converges in distribution to a centered Gaussian random vector with covariance matrix $\Sigma$, and let $g:\mathbb{R}^k \to \R^d$ be continuously differentiable at $\mu$ with Jacobian matrix $J_g$. Then,

\begin{align*}
\sqrt{n} (g(X_n) - g(\mu) ) \overset{d}{\longrightarrow} N\left( 0 , J_g \Sigma J_g^T \right)
\end{align*}

provided $J_g \Sigma J_g^T$ is positive definite.

\end{lemma}

\subsection{Large deviations}

\label{Subsection:LargeDeviations}

Let $X := (X_n)_{n \in \N}$ be a sequence of random elements taking values in some Hausdorff topological space $\mathbb{X}$. Then, we say that $X$ satisfies a large deviation principle (LDP) with speed $s(n)$ and rate function $I_{X}$ if $s:\N\to(0,\infty)$ satisfies $s(n)\to\infty$, $I_X:\mathbb{X}\to[0,\infty]$ is lower semicontinuous, and

\begin{align*}
- \inf _{x \in A^\circ } I_X (x)
&\leq \liminf _{n \to \infty} \frac{1}{s(n)} \log \mathbb{P} (X_n \in A )\\
&\leq \limsup _{n \to \infty} \frac{1}{s(n)} \log \mathbb{P} (X_n \in A )
\leq - \inf _{x \in \overline{ A} } I_X (x)
\end{align*}

for all Borel sets $A \subset \mathbb{X}$.

The rate function $I_X$ is said to be a good rate function if it has compact level sets $\lbrace x \in \mathbb{X} ; I_X (x) \leq \alpha \rbrace$, for every $\alpha \in [0, \infty)$.

Let $\mathbb{X} = \R^d$ for some $d \in \N$ and let $\Lambda : \R^d \to \R \cup \lbrace +\infty \rbrace$. Then, the Legendre-Fenchel transform of $\Lambda$ is denoted as $\Lambda ^* $ and defined as

\begin{align*}
\Lambda^* (x) := \sup _{u \in \R^d} \left[ \langle x , u \rangle - \Lambda (u) \right], \; \; x \in \R^d,
\end{align*}

where $\langle \cdot , \cdot \rangle$ denotes the standard scalar product on $\R^d$. Moreover, we define the (effective) domain of $\Lambda$ to be the set $D_\Lambda := \lbrace u \in \R^d ; \Lambda (u) < \infty \rbrace$.

We now present some fundamental results from large deviations theory. First, we recall Cram\'er's theorem for empirical means of independent and identically distributed random vectors (see Section~2.2 in~\cite{DZ}).

\begin{lemma}[Cramer's theorem]

Let $X_1,X_2,\ldots$ be independent and identically distributed random vectors taking values in $\R^d$. Assume that the origin is an interior point of $D_\Lambda$, where $\Lambda(u)=\log\E e^{\langle u,X_1\rangle}$. Then, the sequence of empirical means $\left(\frac1n\sum_{i=1}^nX_i\right)_{n\in\N}$ satisfies an LDP on $\R^d$ with speed $n$ and good rate function $\Lambda^*$.

\end{lemma}

To treat sums with stretched exponential tails, we use the
following consequence of~\cite[Theorem~1]{GRR}, in the LDP
formulation of~\cite[Lemma~3.4 and Remark~3.5]{KPT}.

\begin{lemma}[Cram\'er's theorem for stretched exponentials]
\label{LemGRR}
Let $X_1,X_2,\ldots$ be non-negative independent and
identically distributed random variables. If
for some $\rho\in(0,1)$ and $\kappa>0$,
\[
\lim_{t\to\infty}\frac1{t^\rho}
\log\mathbb P(X_1\geq t)=-\kappa,
\]
then $(n^{-1}\sum_{i=1}^nX_i)_{n\in\mathbb N}$ satisfies
an LDP on $\mathbb R$ with speed $n^\rho$ and
good rate function
\[
I_X(z)=
\begin{cases}
\kappa(z-\mathbb E X_1)^\rho,
&z\geq\mathbb E X_1,\\
+\infty,&z<\mathbb E X_1.
\end{cases}
\]
\end{lemma}

Next result, which will be used in the following sections, states that if $(X_n)_{n \in \N}$ is a sequence of random variables that satisfies an LDP with speed $s(n)$ and rate function $I_X$ and $(Y_n)_{n \in \N}$ is a sequence of random variables that is "close" to $(X_n)_{n \in \N}$ in some sense, then the sequence $(Y_n)_{n \in \N}$ will also satisfy an LDP under some conditions (Theorem 4.2.13. in \cite{DZ}).

\begin{lemma}[Exponential equivalence]

\label{LemExponentialEquivalence}

Let $X=(X_n)_{n\in\N}$ and $Y=(Y_n)_{n\in\N}$ be two sequences of real-valued random variables defined on a common probability space. Assume that $X$ satisfies an LDP with speed $s(n)\to\infty$ and good rate function $I_X$. Further, assume that $X$ and $Y$ are exponentially equivalent, i.e.,

\begin{align*}
\limsup _{n \to \infty} \frac{1}{s(n)} \log \mathbb{P} ( | X_n - Y_n | > \delta ) = - \infty
\end{align*}

for any $\delta >0$. Then, $Y$ satisfies an LDP with the same speed and the same rate function as $X$.

\end{lemma}

Lastly, the contraction principle allows to transform a given LDP to another by a continuous function (Theorem 4.2.1. in \cite{DZ}).

\begin{lemma}[Contraction principle]

\label{LemContractionPrinciple}

Let $\mathbb{X}$ and $\mathbb{Y}$ be two Hausdorff topological spaces and $F: \mathbb{X} \to \mathbb{Y}$ be a continuous function. Further, let $X = (X_n)_{n \in \N}$ be a sequence of $\mathbb{X}$-valued random elements that satisfies an LDP with speed $s(n)$ and good rate function $I_X$. Then, the sequence $Y:=  ( F(X_n) )_{n \in \N}$ satisfies an LDP on $\mathbb{Y}$ with the same speed and with good rate function $I_Y = I_X \circ F^{-1}$, i.e., $I_Y(y) := \inf \lbrace I_X(x) ; F(x) = y \rbrace$, $y \in \mathbb{Y}$ with the convention that $I_Y(y) = +\infty$ if $F^{-1} ( \lbrace y \rbrace ) = \emptyset$.

\end{lemma}







\section{Previous results and comparison}

\label{section3}

In this section, we will provide a short and compact proof of Hammersley's CLT in the case of the $n$-dimensional Euclidean sphere (see \cite{H}).

\medskip

Let $X^{(n)} , Y^{(n)}$ be independent random vectors uniformly distributed on the Euclidean $n$-dimensional sphere, $S^{n-1}$. In \cite{H}, the author approached the problem by studying the distribution function of the distance between $X^{(n)}$ and $Y^{(n)}$, which was shown to be close to an incomplete beta-function ratio, and therefore, tends asymptotically to a Gaussian, as the dimension grows.
In \cite{L}, the author proved the result by extending it to spherical distributions, whose density function depends only on the modulus of the random vector.

The underlying idea of the proofs relies on the rotational invariance of the sphere, which implies that $\Vert X^{(n)}-Y^{(n)} \Vert _2$ depends only on a single random variable. To see this, fix a point $x_0 \in S^{n-1}$ on the sphere. For any $x\in S^{n-1}$ and $t\geq 0$,

$$ \int _{S^{n-1}} \mathbf{1}_{\{\Vert x-y \Vert _2\leq t\}}\,d\sigma(y) = \int _{S^{n-1}} \mathbf{1}_{\{\Vert x_0 -y \Vert _2\leq t\}}\,d\sigma(y),\qquad t\in\R, $$

where $\sigma$ is the normalized cone measure on $S^{n-1}$, i.e., $\sigma (S^{n-1}) = 1$.
Then, for any $t\geq 0$,

\begin{align*}
\mathbb{P} (\Vert X^{(n)}-Y^{(n)} \Vert _2 \leq t) & = \int_{S^{n-1}} \mathbb{P} (\Vert x-Y^{(n)} \Vert _2 \leq t \mid X^{(n)} = x)  \, d\sigma(x)
 \\
 & =  \int_{S^{n-1}} \int_{S^{n-1}} \mathbf{1}_{\{\Vert x-y \Vert _2 \leq t\}} \, d \sigma(y) \, d\sigma(x)
 \\
 & =  \int_{S^{n-1}}  \mathbf{1}_{\{\Vert x_0 - y \Vert _2 \leq t\}} \, d\sigma(y) \int_{S^{n-1}}  d\sigma(x)
 \\
 & =  \int_{S^{n-1}} \mathbf{1}_{\{\Vert x_0 - y \Vert _2 \leq t\}} \, d\sigma(y)
 \\
 & = \mathbb{P} (\Vert x_0-Y^{(n)} \Vert _2 \leq t).
\end{align*}

Therefore, the problem is reduced to studying the asymptotics of the random variable $\Vert x_0-Y^{(n)} \Vert _2$. For the general case, i.e., when $X^{(n)} , Y^{(n)}$ are uniformly distributed on $B_p^ n$ or $\partial B_p^n$, for $p\in [1, \infty]$ and $n \in \N$, we develop suitable techniques in sections \ref*{section4} and \ref*{section5}.

We will provide a proof of the following theorem.
\begin{theorem}

\label{Thm:CLThypersphere}

Let $X^{(n)},Y^{(n)} \sim \text{Unif} (S^{n-1})$ be independent random vectors. Then,

\begin{align*}
  \frac{ \Vert X^{(n)}-Y^{(n)} \Vert _2 - \mathbb{E} \Vert X^{(n)}-Y^{(n)} \Vert _2}{ \sqrt{ Var\Vert X^{(n)}-Y^{(n)} \Vert _2} } \overset{d}{\to} N(0,1) .
\end{align*}

In particular,

\begin{align*}
\sqrt{2n} \left( \Vert X^{(n)}-Y^{(n)} \Vert _2 - \sqrt{2} \right) \overset{d}{\to}  N(0,1) .
\end{align*}

\end{theorem}
To prove this result, let us first compute the distribution function of the distance.
\begin{lemma}

\label{lem:DistribDistSphere}

Let $X^{(n)},Y^{(n)}$ be independent random vectors uniformly distributed on $S^{n-1}$, for each $n\geq2$. Then, for every $t \in [0,2]$

\begin{align*}
\mathbb{P} (\Vert X^{(n)}-Y^{(n)} \Vert _2 \leq t) & = \frac{\Gamma ( \frac{n}{2} ) }{ \sqrt{\pi} \Gamma ( \frac{n-1}{2} )} \int_{1-\frac{t^2}{2}}^{1} (1- x^2)^{(n-3)/2} dx
 \\
  & = \frac{\Gamma ( \frac{n}{2} ) }{ \sqrt{\pi} \Gamma ( \frac{n-1}{2} )} \int_0^{s}  (\sin (\theta))^{n-2} d\theta ,
\end{align*}

where $s=\arccos(1-t^2/2)$.

\end{lemma}

\begin{proof}
Without loss of generality, we fix a point on the sphere $x_0 = e_1$. Let $\theta \in [0,\pi]$ be the angle between the vectors $x_0$ and $y_\theta$, where $y_\theta$ is a point on the sphere.

First, note that the length of the chord joining $x_0$ and $y_\theta$ is $2 \sin (\theta / 2)$, and so it follows that $\cos( \theta) = 1- \dist(x_0 , y _\theta)^2 / 2 $.

Given an angle $\theta \in (0,\pi)$, the $(n-2)$-dimensional Hausdorff measure of the points at distance $2 \sin (\theta / 2)$ of $x_0$ is
$$
\frac{2 \pi^{\frac{(n-1)}{2}}}{\Gamma ( \frac{n-1}{2} )} (\sin (\theta))^{n-2} .
$$
Notice that, given $t \in [0,2]$, the event $\lbrace \dist (x_0,Y^{(n)}) \leq t \rbrace$ has the same probability as the event $\lbrace 2\sin ( \theta / 2 ) \leq t \rbrace$, where $\theta$ is again the angle between the vectors $x_0$ and $Y^{(n)}$.

Therefore,
\begin{align*}
\mathbb{P} (\Vert X^{(n)}-Y^{(n)} \Vert _2 \leq t) & = \int_0^{s} \frac{\frac{2 \pi^{\frac{n-1}{2}}}{\Gamma ( \frac{n-1}{2} )} (\sin (\theta))^{n-2}}{\frac{2 \pi^{\frac{n}{2}}}{\Gamma ( \frac{n}{2} )}} d\theta
 \\
 & = \frac{\Gamma ( \frac{n}{2} ) }{ \sqrt{\pi} \Gamma ( \frac{n-1}{2} )} \int_0^{s}  (\sin (\theta))^{n-2} d\theta ,
\end{align*}
where $t = 2 \sin (s / 2)$ or $t = \sqrt{2} \sqrt{1- \cos (s)}$, i.e., $s = 2 \arcsin (t / 2)$ or $s = \arccos (1-t^2/2)$. Thus, we can also write is as
\begin{align*}
\mathbb{P} (\Vert X^{(n)}-Y^{(n)} \Vert _2 \leq t) & = \frac{\Gamma ( \frac{n}{2} ) }{ \sqrt{\pi} \Gamma ( \frac{n-1}{2} )} \int_{1-\frac{t^2}{2}}^{1} (1- x^2)^{(n-3)/2} dx.
\end{align*}
\end{proof}

Now, we are ready to compute the expected value of the distance.

\begin{lemma}

\label{Lem:ExpectedValueHypersphere}

Let $X^{(n)},Y^{(n)} \sim \text{Unif}( S^{n-1})$ independent random vectors. Then,

$$ \mathbb{E} \Vert X^{(n)}-Y^{(n)} \Vert _2 = \frac{2^{n-1} \left( \Gamma(n/2) \right)^2}{\sqrt{\pi}\Gamma(n-1/2)}.
$$
Therefore,
$$ 
\mathbb{E} \Vert X^{(n)}-Y^{(n)} \Vert _2 = \sqrt{2} - \frac{1}{4 \sqrt{2}n} + O\left( \frac{1}{n^2} \right) .
$$
\end{lemma}

\begin{proof}
By differentiating the distribution function of $\Vert X^{(n)}-Y^{(n)} \Vert _2$, we obtain its density function given by
\begin{align*}
f(t) = \frac{\Gamma ( \frac{n}{2} ) }{ \sqrt{\pi} \Gamma ( \frac{n-1}{2} )} \left( \sin \left( 2 \arcsin \left( \frac{t}{2} \right) \right) \right)^{n-2} \frac{1}{\sqrt{1-\frac{t^2}{4}}} ,
\end{align*}
for $0<t<2$, with $f(t)=0$ outside this interval.
Applying the change of variables $\theta = 2 \arcsin ( t/2 )$, we obtain that
\begin{align*}
\mathbb{E}\Vert x_0-Y^{(n)} \Vert _2 & = \frac{\Gamma ( \frac{n}{2} ) }{ \sqrt{\pi} \Gamma ( \frac{n-1}{2} )} \int_0^{\pi} 2\sin{(\theta/2)} (\sin (\theta))^{n-2} d\theta  .
\end{align*}
Using the identity $\sin (x) = 2 \sin (x/2) \cos (x/2)$ and the change of variables $x=\theta/2$,
\begin{align*}
\int_0^{\pi} 2\sin{(\theta/2)} (\sin (\theta))^{n-2} d\theta & =
\int_0^{\pi} 2\sin{(\theta/2)} (2 \sin (\theta/2) \cos (\theta/2))^{n-2} d\theta
 \\
 & = \int_0^{\pi} ( 2\sin{(\theta/2)} ) ^{n-1} ( \cos ( \theta /2))^{n-2} d\theta
 \\
 & = 2^{n} \int_0^{\frac{\pi}{2}} ( \sin ( x ) ) ^{n-1} ( \cos (x))^{n-2} dx
  \\
  & = 2^{n-1} B \left( \frac{n}{2} , \frac{n-1}{2}  \right),
\end{align*}
where $B$ is the Beta function. Thus, we conclude that
\begin{align*}
\mathbb{E}\Vert x_0-Y^{(n)} \Vert _2 & = \frac{\Gamma ( \frac{n}{2} ) }{ \sqrt{\pi} \Gamma ( \frac{n-1}{2} )} 2^{n-1} B \left( \frac{n}{2} , \frac{n-1}{2}  \right)
 \\
 & = \frac{\Gamma ( \frac{n}{2} ) }{ \sqrt{\pi} \Gamma ( \frac{n-1}{2} )} 2^{n-1}
 \frac{\Gamma ( \frac{n}{2} ) \Gamma ( \frac{n-1}{2} ) }{ \Gamma ( \frac{2n - 1}{2} )}
 \\
 & = \frac{2^{n-1} \left( \Gamma(n/2) \right)^2}{\sqrt{\pi}\Gamma(n-1/2)}.
\end{align*}
Furthermore, using the second order Stirling's formula,
\begin{align*}
\Gamma (1+n) = \sqrt{2 \pi n } \left( \frac{n}{e} \right)^n \left( 1 + \frac{1}{12n} + O\left( \frac{1}{n^2} \right) \right),
\end{align*}
and applying it to $\Gamma(n/2)$ and $\Gamma ( n- 1/2 )$, a straightforward computation yields  $$\mathbb{E}\Vert x_0-Y^{(n)} \Vert _2 = \sqrt{2} - \frac{1}{4 \sqrt{2}n} + O\left( \frac{1}{n^2} \right) .$$
\end{proof}

In a similar way, we can compute the variance of the distance and its asymptotics.

\begin{lemma}

\label{Lem:VarianceHypersphere}

Let $X^{(n)},Y^{(n)} \sim \text{Unif}( S^{n-1})$ independent random vectors. Then,

$$ Var \Vert X^{(n)}-Y^{(n)} \Vert _2 = 2 - \left( \frac{2^{n-1} \left( \Gamma(n/2) \right)^2}{\sqrt{\pi}\Gamma(n-1/2)} \right)^{2} , $$

and

$$  \lim _{n \to \infty} 2 n  Var \Vert X^{(n)}-Y^{(n)} \Vert _2 = 1. $$

\end{lemma}

\begin{proof}
By symmetry, $\E X^{(n)}=\E Y^{(n)}=0$. Since both points have Euclidean norm $1$ and are independent,
\begin{align*}
\E \Vert X^{(n)}-Y^{(n)} \Vert _2^2
&=\E\Vert X^{(n)}\Vert_2^2+\E\Vert Y^{(n)}\Vert_2^2
-2\langle\E X^{(n)},\E Y^{(n)}\rangle\\
&=2.
\end{align*}
Taking into account the expected value of $\Vert X^{(n)}-Y^{(n)}\Vert_2$ computed in Lemma~\ref{Lem:ExpectedValueHypersphere} we obtain the stated variance. Moreover, the expansion proved there yields
\begin{align*}
\var\Vert X^{(n)}-Y^{(n)} \Vert _2
&=2-\left(\sqrt2-\frac{1}{4\sqrt2\,n}+O\left(\frac{1}{n^{2}}\right)\right)^2\\
&=\frac{1}{2n}+O\left(\frac{1}{n^{2}}\right).
\end{align*}
Thus $2n\,\var\Vert X^{(n)}-Y^{(n)} \Vert _2\to1$.
\end{proof}

Before proving Theorem \ref{Thm:CLThypersphere}, we need the following technical result (see \cite{S}, Theorem).

\begin{theorem}

\label{UniformIntegrabTheorem}

Let $( p_n )_{n \in \N}$ be a sequence of densities on $\R$ such that,

\begin{align*}
\lim _{n \to \infty} p_n (x) = p (x)
\end{align*}

for almost every $x \in \R$. If $p$ is a density, then

\begin{align*}
\lim _{n \to \infty} 	\int_ S p_n(x) dx = 	\int_ S p(x) dx
\end{align*}

uniformly for all Borel sets $S$ in $\R$.

\end{theorem}

Now, we are ready to prove Theorem \ref{Thm:CLThypersphere}.

\begin{proof}[Proof of Theorem~\ref{Thm:CLThypersphere}]
For $n\geq4$, put
\[
f_n(y):=
\begin{cases}
\displaystyle
\frac{\Gamma(n/2)}{\sqrt\pi\,\Gamma((n-1)/2)\sqrt n}
\left(1-\frac{y^2}{n}\right)^{(n-3)/2},
& |y|\leq\sqrt n,\\[2mm]
0,&|y|>\sqrt n.
\end{cases}
\]
By Lemma~\ref{lem:DistribDistSphere} and the change of
variables $y=\sqrt n\,x$, we have that
\[
\int_{\R}f_n(y)\,dy
=\frac{\Gamma(n/2)}{\sqrt\pi\,\Gamma((n-1)/2)}
\int_{-1}^{1}(1-x^2)^{(n-3)/2}\,dx=1,
\]
so $f_n$ is a probability density. For every fixed
$y\in\R$, using Stirling's formula and the logarithmic expansion we obtain that
\begin{align*}
\frac{\Gamma(n/2)}{\sqrt\pi\,\Gamma((n-1)/2)\sqrt n}
&\longrightarrow\frac1{\sqrt{2\pi}},
\end{align*}
and
\begin{align*}
\frac{n-3}{2}\log\left(1-\frac{y^2}{n}\right)
&=-\frac{y^2}{2}+O(n^{-1}).
\end{align*}
Thus $f_n(y)\to(2\pi)^{-1/2}e^{-y^2/2}$.

\medskip

Now, fix $t\in\R$. For all sufficiently large $n$, we have
$\sqrt2+t/\sqrt{2n}\in[0,2]$. Applying
Lemma~\ref{lem:DistribDistSphere} and the same change of
variables, we obtain that
\begin{align*}
&\mathbb P\left(
\sqrt{2n}\bigl(\Vert X^{(n)}-Y^{(n)} \Vert _2-\sqrt2\bigr)
\leq t\right)\\
&\qquad=\int_{-t-t^2/(4\sqrt n)}^\infty f_n(y)\,dy\\
&\qquad=\int_{-t}^\infty f_n(y)\,dy
+\int_{-t-t^2/(4\sqrt n)}^{-t}f_n(y)\,dy.
\end{align*}
Since $0\leq(1-y^2/n)^{(n-3)/2}\leq1$ for
$|y|\leq\sqrt n$, the second integral satisfies
\begin{align*}
0\leq\int_{-t-t^2/(4\sqrt n)}^{-t}f_n(y)\,dy
&\leq
\frac{\Gamma(n/2)}{\sqrt\pi\,\Gamma((n-1)/2)\sqrt n}
\int_{-t-t^2/(4\sqrt n)}^{-t}1\,dy\\
&=\frac{\Gamma(n/2)}{\sqrt\pi\,\Gamma((n-1)/2)\sqrt n}
\frac{t^2}{4\sqrt n}\longrightarrow0.
\end{align*}
For the first integral, Theorem~\ref{UniformIntegrabTheorem}
gives
\[
\int_{-t}^\infty f_n(y)\,dy
\longrightarrow
\frac1{\sqrt{2\pi}}\int_{-t}^\infty e^{-y^2/2}\,dy
=\frac1{\sqrt{2\pi}}\int_{-\infty}^{t}e^{-y^2/2}\,dy
=\Phi(t).
\]
Consequently,
\[
\sqrt{2n}\bigl(\Vert X^{(n)}-Y^{(n)} \Vert _2-\sqrt2\bigr)
\overset d\longrightarrow N(0,1).
\]

Finally, write
$\mu_n:=\E\Vert X^{(n)}-Y^{(n)} \Vert _2$ and
$\sigma_n^2:=\var\Vert X^{(n)}-Y^{(n)} \Vert _2$.
By Lemmas~\ref{Lem:ExpectedValueHypersphere}
and~\ref{Lem:VarianceHypersphere}, we have that $\sqrt{2n}\,\sigma_n\longrightarrow1$ and
\[
\sqrt{2n}(\mu_n-\sqrt2)
=-\frac1{4\sqrt n}+O\left(\frac{1}{n^{3/2}}\right)\longrightarrow0.
\]
Thus, by Slutsky's theorem, we obtain that
\begin{align*}
\frac{\Vert X^{(n)}-Y^{(n)} \Vert _2-\mu_n}{\sigma_n}
&=\frac{
\sqrt{2n}\bigl(\Vert X^{(n)}-Y^{(n)} \Vert _2-\sqrt2\bigr)
-\sqrt{2n}(\mu_n-\sqrt2)}{\sqrt{2n}\,\sigma_n}\\
&\overset d\longrightarrow N(0,1).
\end{align*}
\end{proof}


\section{Central Limit Theorems}
\label{section4}

In this section, we prove Theorems \ref{Thm:CLTforlpballs} and \ref{Thm:CLTforCube}. We first consider two independent points on the boundary of $B_p^n$, with $1\leq p<\infty$, and write their normalised distance as a function of five empirical means. The multivariate central limit theorem and the delta method then give the CLT for the boundary, from which we obtain the result for the ball by using its radial representation. We finish with the cube, where the independence of the coordinates gives a direct proof for points in the ball.

\subsection{CLT for $l_p^n$-balls}

Fix $1\leq p<\infty$, and let $(g_i)_{i\in\mathbb N}$ and $(g_i')_{i\in\mathbb N}$ be two independent sequences of independent random variables with density
\begin{equation}
f_p(x)=\frac{1}{2\Gamma(1+1/p)}e^{-|x|^p},\qquad x\in\mathbb R.
\label{EqGeneratorDensity}
\end{equation}
Recall that
\[
\mathbb E|g_i|^\alpha=M_p(\alpha)
=\frac{\Gamma((\alpha+1)/p)}{\Gamma(1/p)},\qquad \alpha\geq0,
\qquad
c_p=\sqrt{2p^{2/p}M_p(2)}.
\]
In particular, $M_p(p)=1/p$. Put $G=(g_i)_{i=1}^n$ and $G'=(g_i')_{i=1}^n$. By the Schechtman--Zinn representation (Lemma \ref{Lem:SchAndZinn}), the vectors $G/\|G\|_p$ and $G'/\|G'\|_p$ are independent and distributed on $\partial B_p^n$ according to the cone measure, so their normalised distance is
\[
D_{n,p}^{\partial}:=n^{1/p-1/2}
\left\|\frac{G}{\|G\|_p}-\frac{G'}{\|G'\|_p}\right\|_2.
\]

Let us define
\begin{equation}\label{EqFiveMeans}
S_n=(S_{1,n},S_{2,n},S_{3,n},S_{4,n},S_{5,n}):=\frac1n\sum_{i=1}^n
\big(g_i^2,(g_i')^2,g_i g_i',|g_i|^p,|g_i'|^p\big),
\end{equation}
and the function
\[
f(a,b,c,u,v):=
\left(au^{-2/p}+bv^{-2/p}-2c(uv)^{-1/p}\right)^{1/2},
\]
for $u,v>0$ and $au^{-2/p}+bv^{-2/p}-2c(uv)^{-1/p}\geq0$. Expanding the square of the Euclidean distance, we obtain
\begin{equation}
D_{n,p}^{\partial}=f(S_n)
=\left(S_{1,n}S_{4,n}^{-2/p}+S_{2,n}S_{5,n}^{-2/p}
-2S_{3,n}(S_{4,n}S_{5,n})^{-1/p}\right)^{1/2}.
\label{EqDistanceFiveMeans}
\end{equation}

\begin{proof}[Proof of Theorem \ref{Thm:CLTforlpballs}]
The summands in~\eqref{EqFiveMeans} are independent and identically distributed, have finite second moments, and have mean
\[
\mu_p:=\left(M_p(2),M_p(2),0,\frac1p,\frac1p\right).
\]
Let $\Sigma_p$ be their covariance matrix. The multivariate central limit theorem gives
\begin{equation}
\sqrt n(S_n-\mu_p)\overset d\longrightarrow N(0,\Sigma_p).
\label{EqCLTFiveMeans}
\end{equation}
To compute $\Sigma_p$, we use independence, symmetry and the moment identities
\[
M_p(p+2)=\frac3pM_p(2),\qquad
M_p(2p)=\frac{p+1}{p^2}.
\]
In particular,
\begin{align*}
\operatorname{Var}(g_1^2)&=M_p(4)-M_p(2)^2,
&\operatorname{Var}(|g_1|^p)&=\frac1p,\\
\operatorname{Cov}(g_1^2,|g_1|^p)&=\frac{2M_p(2)}p,
&\operatorname{Var}(g_1g_1')&=M_p(2)^2.
\end{align*}
The same identities hold for $g_1'$, while independence and symmetry make all the other off-diagonal covariances zero. Thus,
\[
\Sigma_p=
\begin{pmatrix}
M_p(4)-M_p(2)^2&0&0&\frac{2M_p(2)}p&0\\
0&M_p(4)-M_p(2)^2&0&0&\frac{2M_p(2)}p\\
0&0&M_p(2)^2&0&0\\
\frac{2M_p(2)}p&0&0&\frac1p&0\\
0&\frac{2M_p(2)}p&0&0&\frac1p
\end{pmatrix}.
\]

At $\mu_p$, the function $f$ satisfies $f(\mu_p)=c_p$. Note that the function $f$ defined by
$$f(a,b,c,u,v)=
\left(au^{-2/p}+bv^{-2/p}-2c(uv)^{-1/p}\right)^{1/2}$$ is continuously differentiable in a neighborhood of $\mu _p$, since $u,v = 1/p >0$, and the expresion under the square root is $2p^{2/p} M_p(2) >0$. Its its gradient is
\[
\nabla f(\mu_p)
=\frac{p^{2/p}}{2c_p}
\big(1,1,-2,-2M_p(2),-2M_p(2)\big).
\]
Multiplying this row vector by $\Sigma_p$ and its transpose gives
\begin{align*}
\nabla f(\mu_p)\Sigma_p\nabla f(\mu_p)^T
&=\frac{p^{4/p}}{4c_p^2}
\left[2\big(M_p(4)-M_p(2)^2\big)+4M_p(2)^2
-\frac{8M_p(2)^2}{p}\right]\\
&=\frac{p^{2/p}}{4M_p(2)}
\left[M_p(4)+\left(1-\frac4p\right)M_p(2)^2\right]
=\sigma_p^2.
\end{align*}
Note that the variance is positive. Indeed, the third coordinate $g_1g_1'$ is uncorrelated with the other four coordinates, has positive variance, and its coefficient in $\nabla f(\mu_p)$ is nonzero. Applying the delta method (Lemma~2.2) to~\eqref{EqCLTFiveMeans}, and using $D_{n,p}^{\partial}=f(S_n)$ and $f(\mu_p)=c_p$, we conclude that
\[
\sqrt n\big(D_{n,p}^{\partial}-c_p\big)
\overset d\longrightarrow N(0,\sigma_p^2).
\]

\medskip

To pass to the ball, let $U,U'$ be independent uniform random variables on $[0,1]$, independent of $G,G'$. By the radial representation, the vectors $U^{1/n}G/\|G\|_p$ and $(U')^{1/n}G'/\|G'\|_p$ are independent uniform points in $B_p^n$, and their normalised distance is
\[
D_{n,p}^{B}:=n^{1/p-1/2}
\left\|U^{1/n}\frac{G}{\|G\|_p}-(U')^{1/n}\frac{G'}{\|G'\|_p}\right\|_2.
\]
The reverse triangle inequality gives
\begin{align*}
|D_{n,p}^{B}-D_{n,p}^{\partial}|
\leq (1-U^{1/n})n^{1/p-1/2}\frac{\|G\|_2}{\|G\|_p} +(1-(U')^{1/n})n^{1/p-1/2}\frac{\|G'\|_2}{\|G'\|_p}.
\end{align*}
The Strong Law of Large Numbers and the well known inequality $1-e^{-x}\leq x$ for $x\geq0$, gives
\begin{align*}
n^{1/p-1/2}\frac{\|G\|_2}{\|G\|_p}
&=\frac{S_{1,n}^{1/2}}{S_{4,n}^{1/p}}
\longrightarrow p^{1/p}\sqrt{M_p(2)},
\end{align*}
almost surely, and
\begin{align*}
0\leq\sqrt n(1-U^{1/n})
&\leq\frac{-\log U}{\sqrt n}\longrightarrow0
\end{align*}
almost surely, and the same limits hold for $G'$ and $U'$. Thus $\sqrt n\,|D_{n,p}^{B}-D_{n,p}^{\partial}|\to0$ almost surely under this coupling, and Slutsky's theorem gives the CLT for the ball.
\end{proof}

\subsection{CLT for the cube}
\label{Subsection:CLTCube}

For the cube, we use the following coupling, which will also be useful for large deviations in the next Section.

\begin{lemma}
\label{LemCubeCoupling}
There exists a coupling of two independent uniform points $X^{(n)},Y^{(n)}$ in $B_\infty^n$ with two independent cone-distributed points $\widetilde X^{(n)},\widetilde Y^{(n)}$ on $\partial B_\infty^n$ such that
\begin{equation}
\left|\frac1n\|X^{(n)}-Y^{(n)}\|_2^2
-\frac1n\|\widetilde X^{(n)}-\widetilde Y^{(n)}\|_2^2\right|
\leq\frac8n
\qquad\text{almost surely}.
\label{EqCubeCoupling}
\end{equation}
\end{lemma}

\begin{proof}
Write $X^{(n)}=(u_1,\ldots,u_n)$ and $Y^{(n)}=(u_1',\ldots,u_n')$, where all the coordinates are independent and uniform on $[-1,1]$. Let $K_X,K_Y$ be uniform on $\{1,\ldots,n\}$ and let $\varepsilon_X,\varepsilon_Y$ be uniform on $\{-1,1\}$, all mutually independent and independent of the coordinates. Put
\[
\widetilde X_i^{(n)}=
\begin{cases}
\varepsilon_X,&i=K_X,\\
u_i,&i\ne K_X,
\end{cases}
\qquad
\widetilde Y_i^{(n)}=
\begin{cases}
\varepsilon_Y,&i=K_Y,\\
u_i',&i\ne K_Y.
\end{cases}
\]
For a Borel set $A$ contained in a facet of the cube, its cone probability is
\[
\frac{\operatorname{vol}_n(\{tx:x\in A,\ 0\leq t\leq1\})}
     {\operatorname{vol}_n(B_\infty^n)}
=\frac{\operatorname{vol}_{n-1}(A)}{n2^n}
=\frac1{2n}\,
  \frac{\operatorname{vol}_{n-1}(A)}{2^{n-1}}.
\]
Thus the cone measure assigns probability $1/(2n)$ to each facet
and, conditionally on that facet, is its normalised
$(n-1)$-dimensional Lebesgue measure. Our construction has precisely
this distribution: the pair $(K_X,\varepsilon_X)$ selects each facet
with probability $1/(2n)$, while the remaining coordinates are
independent and uniform on $[-1,1]$. The same argument applies to
$\widetilde Y^{(n)}$.

To compare the distances, observe that
$\widetilde X_i^{(n)}=u_i$ and $\widetilde Y_i^{(n)}=u_i'$
whenever $i\notin\{K_X,K_Y\}$. Consequently,
\begin{align*}
&\left|\frac1n\|X^{(n)}-Y^{(n)}\|_2^2
-\frac1n\|\widetilde X^{(n)}-\widetilde Y^{(n)}\|_2^2\right|\\
&\qquad=
\frac1n\left|
\sum_{i\in\{K_X,K_Y\}}
\left[(u_i-u_i')^2
-(\widetilde X_i^{(n)}-\widetilde Y_i^{(n)})^2\right]
\right|\\
&\qquad\leq
\frac1n\sum_{i\in\{K_X,K_Y\}}
\left|(u_i-u_i')^2
-(\widetilde X_i^{(n)}-\widetilde Y_i^{(n)})^2\right| \leq\frac8n.
\end{align*}
Here we used that all coordinates belongs to $[-1,1]$, so both
squared differences in each summand lie in $[0,4]$.
This proves \eqref{EqCubeCoupling}.
\end{proof}

\begin{proof}[Proof of Theorem \ref{Thm:CLTforCube}]
For two independent uniform  points in $B_\infty^n$, $X^{(n)}$, $Y^{(n)}$, use the coordinates from Lemma~\ref{LemCubeCoupling} and write their normalised distance as
\[
D_{n,\infty}^{B}:=\left(\frac1n\sum_{i=1}^n|u_i-u_i'|^2\right)^{1/2}.
\]
By symmetry and independence,
\begin{align*}
\mathbb E|u_i-u_i'|^2&=\frac23,\\
\mathbb E|u_i-u_i'|^4
&=2\mathbb E u_i^4+6(\mathbb E u_i^2)^2
=\frac25+\frac23=\frac{16}{15}.
\end{align*}
Thus $\operatorname{Var}(|u_i-u_i'|^2)=28/45$, and the classical central limit theorem gives
\[
\sqrt n\left((D_{n,\infty}^{B})^2-\frac23\right)
\overset d\longrightarrow N\left(0,\frac{28}{45}\right).
\]
Let $f(x)=\sqrt x$ for $x>0$. Since $D_{n,\infty}^{B}=f((D_{n,\infty}^{B})^2)$ and $f'(2/3)=1/(2\sqrt{2/3})$, the delta method (Lemma \ref{DeltaMethod}) yields
\[
\sqrt n\left(D_{n,\infty}^{B}-\sqrt{\frac23}\right)
\overset{d}{\longrightarrow}
N\left(0,\frac{28/45}{4(2/3)}\right)
=N\left(0,\frac7{30}\right).
\]
 
Let now $\tilde{X}^{(n)}$, $\tilde{Y}^{(n)}$ be the associated coupling provided by Lemma \ref{LemCubeCoupling}, and let $D_{n,\infty}^{\partial}=n^{-1/2}
\|\widetilde X^{(n)}-\widetilde Y^{(n)}\|_2$ be the normalised distane between them.
The Strong Law of Large Numbers gives $D_{n,\infty}^{B}\to\sqrt{2/3}$ almost surely, and then Lemma \ref{LemCubeCoupling} implies
\begin{align*}
\sqrt n\,|D_{n,\infty}^{\partial}-D_{n,\infty}^{B}|
&=\frac{\sqrt n\,|(D_{n,\infty}^{\partial})^2-(D_{n,\infty}^{B})^2|}
{D_{n,\infty}^{\partial}+D_{n,\infty}^{B}}\\
&\leq\frac8{\sqrt n\,D_{n,\infty}^{B}}
\longrightarrow0\qquad\text{almost surely}.
\end{align*}
Then, applying Slutsky's theorem, we obtain the CLT on the boundary.
\end{proof}

\section{Large Deviation Principles for $2\leq p\leq\infty$}
\label{section5}

In this section, we are going to prove the Large Deviation Principles for $2\leq p \leq \infty$ by using Cram\'er's theorem, the contraction principle and exponential equivalence. We devide the section into two subsections. The first one is devoted to random vectors uniformly distributed on $B_p^n$ and $\partial B_p^n$ ($2\leq p <\infty$) in order to prove Theorem~\ref{ThmLDPpgeq2}, and the second is devoted to the case of the cube (Theorem \ref{ThmLDPBoundaryCube}).

\subsection{LDP for $l_p^n$-balls, $p\geq2$}

For $2\leq p<\infty$, recall the random vector $S_n$
from~\eqref{EqFiveMeans}. The function $\widetilde\Lambda_p$
defined in~\eqref{EqLogMGFFiveIntro} is the logarithmic
moment generating function of
$(g_1^2,(g_1')^2,g_1g_1',|g_1|^p,|g_1'|^p)$.
Thus,
\begin{align}
\widetilde\Lambda_p(t)
&=\log\E\exp\left(
t_1g_1^2+t_2(g_1')^2+t_3g_1g_1'
+t_4|g_1|^p+t_5|g_1'|^p
\right),
\label{EqLogMGFFive}
\end{align}
for $t=(t_1,\ldots,t_5)\in\R^5$.

\begin{proof}[Proof of Theorem~\ref{ThmLDPpgeq2}, boundary case]
We first apply Cram\'er's theorem to $S_n$. Assuming that $\max_i|t_i|<1/5$, the inequalities $|xy|\leq(x^2+y^2)/2$ and $x^2\leq1+|x|^p$ give
\begin{align*}
&t_1x^2+t_2y^2+t_3xy-(1-t_4)|x|^p-(1-t_5)|y|^p\\
&\qquad\leq\frac35-\frac12(|x|^p+|y|^p).
\end{align*}
Using the density of $g_1$ and $g_1'$, we obtain that
\[
\exp\big(\widetilde\Lambda_p(t)\big)
\leq\frac{e^{3/5}}{4\Gamma(1+1/p)^2}
\left(\int_{\R}e^{-|x|^p/2}\,dx\right)^2<\infty.
\]
Thus $\widetilde\Lambda_p$ is finite in a neighborhood of the origin, and $(S_n)_{n\in\N}$ satisfies an LDP on $\R^5$ with speed $n$ and good rate function $\widetilde\Lambda_p^*$.

\medskip

To apply the contraction principle to $D_{n,p}^\partial$, we verify that $S_n$ belongs almost surely to the set $\mathcal{E}$ introduced in Section $\ref{section1}$,
\begin{align*}
\mathcal E:=\{(a,b,c,u,v)\in\R^5:
u,v>0,\ a,b\geq0,\ c^2\leq ab\}.
\end{align*}

The first two coordinates of $S_n$ are non-negative, and Cauchy--Schwarz gives
\[
S_{3,n}^2=\left(\frac1n\sum_{i=1}^n g_i g_i'\right)^2
\leq\left(\frac1n\sum_{i=1}^n g_i^2\right)
\left(\frac1n\sum_{i=1}^n(g_i')^2\right)
=S_{1,n}S_{2,n}.
\]
Moreover, since $S_{4,n}$ and
$S_{5,n}$ are positive almost surely, then
$S_n\in\mathcal E$ almost surely.

\medskip

We next verify that $\widetilde\Lambda_p^*$ is infinite
outside $\mathcal E$. First, fix
$x\notin\overline{\mathcal E}$ and choose an open ball
$B(x,\delta)$ disjoint from $\overline{\mathcal E}$.
Since $S_n\in\mathcal E$ almost surely,
$\mathbb P(S_n\in B(x,\delta))=0$ for every $n$.
Applying the LDP lower bound to this open ball gives
\[
-\inf_{y\in B(x,\delta)}\widetilde\Lambda_p^*(y)
\leq
\liminf_{n\to\infty}\frac1n
\log\mathbb P(S_n\in B(x,\delta))
=-\infty.
\]
Thus the infimum is infinite, so
$\widetilde\Lambda_p^*$ is infinite throughout the ball,
in particular at $x$.

\medskip

It remains to consider the points of
$\overline{\mathcal E}\setminus\mathcal E$. Since
\[
\overline{\mathcal E}
=\bigl\{(a,b,c,u,v)\in\mathbb R^5:
a,b\geq0,\ c^2\leq ab,\ u,v\geq0\bigr\},
\]
these are precisely the points of $\overline{\mathcal E}$
for which $u=0$ or $v=0$.

For $\lambda>0$, the change of variables
$y=(1+\lambda)^{1/p}x$ gives
\[
\mathbb E e^{-\lambda|g_1|^p}
=\frac{1}{2\Gamma(1+1/p)}
\int_{\mathbb R}e^{-(1+\lambda)|x|^p}\,dx
=(1+\lambda)^{-1/p}.
\]
At a point $(a,b,c,0,v)$, choosing
$t=(0,0,0,-\lambda,0)$ in the supremum defining the
Legendre--Fenchel transform therefore yields
\[
\widetilde\Lambda_p^*(a,b,c,0,v)
\geq-\widetilde\Lambda_p(0,0,0,-\lambda,0)
=\frac1p\log(1+\lambda).
\]
Letting $\lambda\to\infty$ proves that the rate is
infinite whenever $u=0$.
At a point $(a,b,c,u,0)$, the choice
$t=(0,0,0,0,-\lambda)$ gives
\[
\widetilde\Lambda_p^*(a,b,c,u,0)
\geq\frac1p\log(1+\lambda),
\]
so the rate is also infinite whenever $v=0$.
Together with the preceding argument outside
$\overline{\mathcal E}$, this proves that
$\widetilde\Lambda_p^*=\infty$ outside $\mathcal E$.

Thus, $S_n\in\mathcal E$ almost surely and
$\widetilde\Lambda_p^*=\infty$ outside $\mathcal E$. The rate function
is good and its finite sublevel sets are compact
in $\mathbb R^5$ and contained in $\mathcal E$.


\medskip

Recall that $D_{n,p}^{\partial}=f(S_n)$, where
\[
f(a,b,c,u,v)
=\left(
au^{-2/p}+bv^{-2/p}-2c(uv)^{-1/p}
\right)^{1/2}.
\]
For $(a,b,c,u,v)\in\mathcal E$, the inequality
$c\leq\sqrt{ab}$ gives
\[
au^{-2/p}+bv^{-2/p}-2c(uv)^{-1/p}
\geq
\left(\sqrt a\,u^{-1/p}-\sqrt b\,v^{-1/p}\right)^2
\geq0.
\]
Since $u,v>0$, then the function $f$ is continuous on $\mathcal E$. The contraction principle now yields an LDP for $(D_{n,p}^{\partial})_{n\in\mathbb N}$
with speed $n$ and good rate function
\begin{align*}
I_p(z)
=\inf_{\substack{(a,b,c,u,v)\in\mathcal E\\
au^{-2/p}+bv^{-2/p}-2c(uv)^{-1/p}=z^2}}
\widetilde\Lambda_p^*(a,b,c,u,v),
\qquad z\geq0,
\end{align*}
and $I_p(z)=\infty$ for $z<0$, since $f$ takes only a
non-negative values.
\end{proof}

Let us prove now the LDP for the $l_p^ n$-ball

\begin{proof}[Completion of the proof of
Theorem~\ref{ThmLDPpgeq2}, ball case]
Let $2\leq p<\infty$, and let $U,U'$ be independent random
variables uniformly distributed on $[0,1]$, independent of
$G,G'$. By the Schechtman--Zinn representation, we can write
two independent uniform points in $B_p^n$ as
\[
X^{(n)} \overset{d}{=} U^{1/n}\frac{G}{\|G\|_p},
\qquad
Y^{(n)} \overset{d}{=} (U')^{1/n}\frac{G'}{\|G'\|_p}.
\]

As proved in Lemma 3.3 of \cite{GKR}, the sequence $(U^{1/n})_{n \in \N}$ satisfy an LDP on $\R$ with speed $n$ and rate function
\[
I_U(z) =
     \begin{cases}
       - \log (z) &\quad\text{: }z \in (0, 1]\\
       + \infty &\quad\text{: otherwise}. \\
     \end{cases}
\]

Since $S_n,U^{1/n}$ and $(U')^{1/n}$ are independent,
the LDP for $S_n$ established in the boundary proof and
the product LDP gives an LDP for
\[
\bigl(S_n,U^{1/n},(U')^{1/n}\bigr)
\]
on $\mathcal E\times[0,1]^2$, with speed $n$ and good
rate function
\[
\widetilde\Lambda_p^*(a,b,c,u,v)-\log r-\log s
\]
when $r,s>0$, and infinite rate if either radius is zero.

Expanding the squared distance in the representation above
and using the definition of $f$, we obtain
\[
D_{n,p}^{B}
=f\bigl(
U^{2/n}S_{1,n},(U')^{2/n}S_{2,n},
(UU')^{1/n}S_{3,n},S_{4,n},S_{5,n}
\bigr).
\]
For $(a,b,c,u,v)\in\mathcal E$ and $r,s\in[0,1]$,
the vector $(r^2a,s^2b,rsc,u,v)$ belongs to $\mathcal E$. Indeed,
its first two coordinates are non-negative, its last two
coordinates are positive, and
\[
(rsc)^2\leq r^2s^2ab=(r^2a)(s^2b).
\]
The continuity of $f$ on $\mathcal E$ therefore implies
that the map
\[
(a,b,c,u,v,r,s)\longmapsto f(r^2a,s^2b,rsc,u,v)
\]
is continuous on $\mathcal E\times[0,1]^2$.
The contraction principle now gives the required LDP
with speed $n$ and good rate function $I_p^B$ stated
in Theorem~\ref{ThmLDPpgeq2}. 
\end{proof}

\subsection{LDP for the cube}

We now turn to the case $p=\infty$, where the LDP in the ball follows directly from the independence of the coordinates. The coupling established in Lemma~\ref{LemCubeCoupling} will then give the same result for the boundary.

\begin{proof}[Proof of Theorem~\ref{ThmLDPBoundaryCube}]
Let $X^{(n)},Y^{(n)}$ be independent and uniformly distributed in $B_\infty^n$. As in the proof of Theorem~\ref{Thm:CLTforCube}, write $X^{(n)}=(u_1,\ldots,u_n)$ and $Y^{(n)}=(u_1',\ldots,u_n')$, where all the coordinates are independent and uniform on $[-1,1]$, so that
\begin{align*}
(D_{n,\infty}^{B})^2=\frac1n\sum_{i=1}^n|u_i-u_i'|^2.
\end{align*}
The density of $u_1-u_1'$ is $(2-|x|)/4$ on $[-2,2]$, and consequently the density of $|u_1-u_1'|$ is $1-x/2$ on $[0,2]$. Thus,
\begin{align*}
\log\E e^{t|u_1-u_1'|^2}
=\log\left[\int_0^2e^{tx^2}\left(1-\frac x2\right)\,dx\right]
=\Lambda_\infty(t).
\end{align*}
The squared differences take values in $[0,4]$, so $\Lambda_\infty$ is finite on $\R$ and Cram\'er's theorem gives an LDP for $((D_{n,\infty}^{B})^2)_{n\in\N}$ on $[0,4]$, with speed $n$ and good rate function $\Lambda_\infty^*$. Applying the contraction principle to $x\mapsto\sqrt x$ now gives the stated rate $I_\infty$.

\medskip

For the boundary, take $\widetilde{X}^{(n)},\widetilde{Y}^{(n)}\sim \text{Unif}(\partial B_{\infty}^n)$ from Lemma \ref{LemCubeCoupling}. Then, Lemma \ref{LemCubeCoupling} and the inequality $|\sqrt a-\sqrt b|\leq\sqrt{|a-b|}$ for $a,b\geq0$ give
\begin{align*}
\left|n^{-1/2}\Vert\widetilde X^{(n)}-\widetilde Y^{(n)}\Vert_2
-D_{n,\infty}^{B}\right|\leq\sqrt{\frac8n}.
\end{align*}
The bound tends to zero deterministically, which proves exponential equivalence at speed $n$ and gives the same LDP for the boundary.
\end{proof}

\section{Large Deviation Principles for \texorpdfstring{$1\le p<2$}{1 <= p < 2}}
\label{section6}

In this section, we prove Theorem~\ref{ThmLDPpless2}. We keep the notation $G,G'$ and $S_n$ from Section~\ref{section4}. Since $\E\e^{tg_1^2}=\infty$ for every $t>0$, Cram\'er's theorem could not be applied to the squared coordinates. We instead use their stretched exponential tails to obtain an LDP for
\[
T_n:=\frac1n\sum_{i=1}^n(g_i-g_i')^2
=S_{1,n}+S_{2,n}-2S_{3,n},
\]
and then compare $p^{1/p}\sqrt{T_n}$ with the normalised distances in the ball and on its boundary

\begin{lemma}
\label{LemStretchedTail}
Let $1\le p<2$. Then
\[
\lim_{t\to\infty}\frac1{t^p}\log\mathbb P(g_1-g_1'>t)=-2^{1-p}.
\]
Consequently,
\begin{align*}
\lim_{t\to\infty}\frac1{t^{p/2}}\log\mathbb P((g_1-g_1')^2\ge t)&=-2^{1-p},\\
\lim_{t\to\infty}\frac1{t^{p/2}}\log\mathbb P(g_1^2\ge t)&=-1.
\end{align*}
\end{lemma}

\begin{proof}
We first estimate $\mathbb P(g_1>s)$ for $s>0$. From the
density of $g_1$, we obtain
\[
\mathbb P(g_1>s)
=\frac{1}{2\Gamma(1+1/p)}
\int_s^\infty e^{-x^p}\,dx
\geq\frac{1}{2\Gamma(1+1/p)}
\int_s^{s+1}e^{-x^p}\,dx
\geq\frac{e^{-(s+1)^p}}{2\Gamma(1+1/p)}.
\]
For the upper bound, note that, since $p\geq1$,
\[
(s+y)^p-s^p
=p\int_0^y(s+r)^{p-1}\,dr
\geq p\int_0^y r^{p-1}\,dr
=y^p,\qquad y\geq0.
\]
Thus, making the change of variables $x=s+y$, we have
\begin{align*}
\mathbb P(g_1>s)
&=\frac{1}{2\Gamma(1+1/p)}
\int_0^\infty e^{-(s+y)^p}\,dy\\
&\leq\frac{e^{-s^p}}{2\Gamma(1+1/p)}
\int_0^\infty e^{-y^p}\,dy
=\frac12e^{-s^p}.
\end{align*}
Taking logarithms in these two bounds and dividing by
$s^p$, we conclude that
\begin{equation}
\lim_{s\to\infty}\frac{1}{s^p}
\log\mathbb P(g_1>s)=-1,
\label{EqOneSidedGammaTail}
\end{equation}
because $(s+1)^p/s^p\to1$.

\medskip

We next estimate $\mathbb P(g_1-g_1'>t)$. Convexity of
$x\mapsto |x|^p$ gives
\[
\left|\frac{x-y}{2}\right|^p
\leq\frac{|x|^p+|y|^p}{2},
\]
so the event $\{g_1-g_1'>t\}$ is contained in
$\{|g_1|^p+|g_1'|^p>2^{1-p}t^p\}$.
Moreover, for $0<\lambda<1$, the change of variables
$y=(1-\lambda)^{1/p}x$ yields
\begin{align*}
\E e^{\lambda|g_1|^p}
&=\frac{1}{2\Gamma(1+1/p)}
\int_{\mathbb R}e^{-(1-\lambda)|x|^p}\,dx\\
&=\frac{(1-\lambda)^{-1/p}}{2\Gamma(1+1/p)}
\int_{\mathbb R}e^{-|y|^p}\,dy
=(1-\lambda)^{-1/p}.
\end{align*}
The exponential Markov inequality and independence
therefore give
\begin{align*}
\mathbb P(g_1-g_1'>t)
&\leq\mathbb P\bigl(
|g_1|^p+|g_1'|^p>2^{1-p}t^p
\bigr)\\
&\leq e^{-\lambda2^{1-p}t^p}
\E e^{\lambda(|g_1|^p+|g_1'|^p)}\\
&=e^{-\lambda2^{1-p}t^p}
\bigl(\E e^{\lambda|g_1|^p}\bigr)^2\\
&=(1-\lambda)^{-2/p}
e^{-\lambda2^{1-p}t^p}.
\end{align*}
For each fixed $\lambda\in(0,1)$, this implies
\[
\limsup_{t\to\infty}\frac{1}{t^p}
\log\mathbb P(g_1-g_1'>t)
\leq-\lambda2^{1-p}.
\]
Since this inequality holds for every fixed $\lambda\in(0,1)$
and its left-hand side does not depend on $\lambda$,
letting $\lambda\uparrow1$ gives the upper bound $-2^{1-p}$.

For the lower bound, if $g_1>t/2$ and $g_1'<-t/2$,
then $g_1-g_1'>t$. Hence independence and symmetry give
\[
\mathbb P(g_1-g_1'>t)
\geq\mathbb P(g_1>t/2)\mathbb P(g_1'<-t/2)
=\mathbb P(g_1>t/2)^2.
\]
Using equation \eqref{EqOneSidedGammaTail}, we obtain
\begin{align*}
\liminf_{t\to\infty}\frac{1}{t^p}
\log\mathbb P(g_1-g_1'>t)
&\geq
2^{1-p}\lim_{t\to\infty}
\frac{\log\mathbb P(g_1>t/2)}{(t/2)^p}\\
&=-2^{1-p}.
\end{align*}
This lower bound together with the upper bound that we proved, proves the first
limit in the statement.

\medskip

Finally, $g_1$ and $g_1-g_1'$ have symmetric densities,
so, for $t>0$,
\begin{align*}
\mathbb P\bigl((g_1-g_1')^2\geq t\bigr)
&=2\mathbb P(g_1-g_1'>\sqrt t),
\end{align*}
and
\begin{align*}
\mathbb P(g_1^2\geq t)
&=2\mathbb P(g_1>\sqrt t).
\end{align*}
Taking logarithms and using $t^{p/2}=(\sqrt t)^p$,
we obtain
\begin{align*}
\frac{\log\mathbb P((g_1-g_1')^2\geq t)}{t^{p/2}}
&=\frac{\log2}{t^{p/2}}
+\frac{\log\mathbb P(g_1-g_1'>\sqrt t)}{(\sqrt t)^p}
\longrightarrow-2^{1-p}
\end{align*}
and
\begin{align*}
\frac{\log\mathbb P(g_1^2\geq t)}{t^{p/2}}
&=\frac{\log2}{t^{p/2}}
+\frac{\log\mathbb P(g_1>\sqrt t)}{(\sqrt t)^p}
\longrightarrow-1.
\end{align*}
\end{proof}

Now, we are ready to prove Theorem \ref{ThmLDPpless2}

\begin{proof}[Proof of Theorem~\ref*{ThmLDPpless2}]
By symmetry and independence, $\E(g_1-g_1')^2=2M_p(2)$.
Applying Cram\'er's theorem for stretched exponentials
(Lemma~\ref*{LemGRR}) with $\rho=p/2$ and
$\kappa=\kappa_p:=2^{1-p}$, using the tail estimate in
Lemma~\ref*{LemStretchedTail}, gives an LDP for $T_n$
with speed $n^{p/2}$ and good rate function
\begin{equation}
I_T(z)=\begin{cases}
\kappa_p(z-2M_p(2))^{p/2},&z\geq2M_p(2),\\
+\infty,&z<2M_p(2).
\end{cases}
\label{EqRateTn}
\end{equation}
Applying the contraction principle
(Lemma~\ref*{LemContractionPrinciple}) for the continuous
function $x\mapsto p^{1/p}\sqrt x$ on $[0,\infty)$ gives
an LDP for $p^{1/p}\sqrt{T_n}$ with the same speed and
rate function
\[
I_T(z^2p^{-2/p})
=\kappa_p\bigl(z^2p^{-2/p}-2M_p(2)\bigr)^{p/2}
=\frac{\kappa_p}{p}(z^2-c_p^2)^{p/2},
\qquad z\geq c_p,
\]
with infinite rate for $z<c_p$. This is the rate function
in the statement, so it remains to prove exponential
equivalence for the two distances.

\medskip

Let $U,U'$ be independent uniform random variables on
$[0,1]$, independent of $G,G'$. By the representations
used in Section~\ref*{section4}, we may realise the two
distances as
\begin{align*}
D^\partial_{n,p}
&=\frac1{\sqrt n}
\left\|S_{4,n}^{-1/p}G-S_{5,n}^{-1/p}G'\right\|_2,\\
D^B_{n,p}
&=\frac1{\sqrt n}
\left\|U^{1/n}S_{4,n}^{-1/p}G
-(U')^{1/n}S_{5,n}^{-1/p}G'\right\|_2.
\end{align*}
We want to compare both of them with the quantity
\[
p^{1/p}\sqrt{T_n}
=\frac1{\sqrt n}\|p^{1/p}G-p^{1/p}G'\|_2.
\]
For that, we first have to control the errors made by
replacing $S_{4,n}^{-1/p}$ and $S_{5,n}^{-1/p}$ by
$p^{1/p}$.

For $\lambda<1$, the moment generating function of
$|g_1|^p$ gives
\[
\log\E e^{\lambda|g_1|^p}
=-\frac1p\log(1-\lambda).
\]
Since this function is finite in a neighbourhood of the
origin, Cram\'er's theorem (Lemma~2.3) gives an LDP for
$S_{4,n}$ with speed $n$ and a good rate function whose unique zero is $\E|g_1|^p=1/p$. Therefore, for every $\eta>0$ there exists $c_\eta>0$ such that, for all
sufficiently large $n$,
\[
\mathbb P\bigl(|S_{4,n}-1/p|\geq\eta\bigr)
\leq2e^{-c_\eta n}.
\]
Now, given $\epsilon>0$, continuity of $x\mapsto x^{-1/p}$
at $1/p$ allows us to choose $0<\eta<1/p$ so that
$|x-1/p|<\eta$ implies $|x^{-1/p}-p^{1/p}|<\epsilon$.
Hence
\[
\mathbb P\bigl(|S_{4,n}^{-1/p}-p^{1/p}|>\epsilon\bigr)
\leq\mathbb P\bigl(|S_{4,n}-1/p|\geq\eta\bigr)
\leq2e^{-c_\eta n}.
\]
Since $p<2$ and
$n^{-p/2}\log(2e^{-c_\eta n})
=n^{-p/2}\log2-c_\eta n^{1-p/2}\to-\infty$,
we obtain that
\begin{equation}
\limsup_{n\to\infty}\frac1{n^{p/2}}
\log\mathbb P\bigl(|S_{4,n}^{-1/p}-p^{1/p}|>\epsilon\bigr)
=-\infty.
\label{EqNormalizerSuperexponential}
\end{equation}

Since $0\leq U^{1/n}\leq1$, the triangle inequality gives
\[
|U^{1/n}S_{4,n}^{-1/p}-p^{1/p}|
\leq|S_{4,n}^{-1/p}-p^{1/p}|
+p^{1/p}(1-U^{1/n}).
\]
If the left-hand side exceeds $\epsilon$, at least one
of the two terms on the right exceeds $\epsilon/2$.
Consequently,
\begin{align*}
&\mathbb P\bigl(
|U^{1/n}S_{4,n}^{-1/p}-p^{1/p}|>\epsilon\bigr)\\
&\qquad\leq
\mathbb P\bigl(|S_{4,n}^{-1/p}-p^{1/p}|>\epsilon/2\bigr)
+\mathbb P\left(1-U^{1/n}>\frac{\epsilon}{2p^{1/p}}\right)\\
&\qquad\leq
\mathbb P\bigl(|S_{4,n}^{-1/p}-p^{1/p}|>\epsilon/2\bigr)
+e^{-\epsilon n/(2p^{1/p})}.
\end{align*}
For the last inequality, we used
$\mathbb P(1-U^{1/n}>a)=(1-a)^n\leq e^{-an}$ when
$0<a<1$; when $a\geq1$, this probability is zero.
Combining this estimate with~\eqref{EqNormalizerSuperexponential},
applied with $\epsilon/2$, and using $p<2$, we obtain that
\begin{equation}
\limsup_{n\to\infty}\frac1{n^{p/2}}
\log\mathbb P\bigl(
|U^{1/n}S_{4,n}^{-1/p}-p^{1/p}|>\epsilon\bigr)
=-\infty.
\label{EqRadialNormalizerSuperexponential}
\end{equation}
Note that the estimates~\eqref{EqNormalizerSuperexponential}
and~\eqref{EqRadialNormalizerSuperexponential} also hold
with $S_{5,n}$ in place of $S_{4,n}$ and $U'$ in place
of $U$, since the pairs $(S_{4,n},U)$ and $(S_{5,n},U')$
have the same distribution.

\medskip

It remains to transfer these coefficient estimates to
the distances themselves and prove their exponential equivalence with $p^{1/p}\sqrt{T_n}$. Using the reverse triangle inequality and the triangle inequality, we
obtain that
\begin{align*}
|D^B_{n,p}-p^{1/p}\sqrt{T_n}|
\leq{}&|U^{1/n}S_{4,n}^{-1/p}-p^{1/p}|\sqrt{S_{1,n}}\\
&+|(U')^{1/n}S_{5,n}^{-1/p}-p^{1/p}|\sqrt{S_{2,n}},
\end{align*}
where we used $\|G\|_2/\sqrt n=\sqrt{S_{1,n}}$ and
$\|G'\|_2/\sqrt n=\sqrt{S_{2,n}}$.

Fix $\delta>0$ and $\epsilon>0$. If both coefficient
errors are at most $\epsilon$ and both $S_{1,n}$ and
$S_{2,n}$ are at most $\delta^2/(4\epsilon^2)$, then
the preceding distance bound is at most
$\epsilon\delta/(2\epsilon)+\epsilon\delta/(2\epsilon)=\delta$.
Therefore, if the distance error exceeds $\delta$, at
least one of these four inequalities fails. Then, we
have that
\begin{align*}
&\mathbb P\bigl(|D^B_{n,p}-p^{1/p}\sqrt{T_n}|>\delta\bigr)\\
&\qquad\leq
2\mathbb P\bigl(|U^{1/n}S_{4,n}^{-1/p}-p^{1/p}|>\epsilon\bigr)
+2\mathbb P\left(S_{1,n}>\frac{\delta^2}{4\epsilon^2}\right).
\end{align*}
Here we used that $(U,S_{4,n})$ and $(U',S_{5,n})$
have the same distribution, and also $S_{1,n}$ and
$S_{2,n}$ have the same distribution.

For fixed $\epsilon>0$, the first probability is
controlled by~\eqref{EqRadialNormalizerSuperexponential}.
To estimate the second probability, we apply Cram\'er's theorem for stretched exponentials (Lemma~\ref*{LemGRR})
to $g_i^2$, using the tail estimate in Lemma~\ref*{LemStretchedTail}, with $\rho=p/2$,
$\kappa=1$ and $\E g_i^2=M_p(2)$. This gives an LDP
for $S_{1,n}$ with speed $n^{p/2}$ and rate
$(x-M_p(2))^{p/2}$ for $x\geq M_p(2)$, and infinite
rate for $x<M_p(2)$. Then, we choose
$0<\epsilon<\delta/(2\sqrt{M_p(2)})$, so that the LDP
upper bound on the closed set
$[\delta^2/(4\epsilon^2),\infty)$ gives
\begin{equation}
\limsup_{n\to\infty}\frac1{n^{p/2}}
\log\mathbb P\left(S_{1,n}>\frac{\delta^2}{4\epsilon^2}\right)
\leq-\left(\frac{\delta^2}{4\epsilon^2}-M_p(2)\right)^{p/2}.
\label{EqSquareMeanTail}
\end{equation}
Combining~\eqref{EqRadialNormalizerSuperexponential}
and~\eqref{EqSquareMeanTail} in the preceding probability
bound, we obtain that
\[
\limsup_{n\to\infty}\frac1{n^{p/2}}
\log\mathbb P\bigl(|D^B_{n,p}-p^{1/p}\sqrt{T_n}|>\delta\bigr)
\leq-\left(\frac{\delta^2}{4\epsilon^2}-M_p(2)\right)^{p/2}.
\]
The left-hand side does not depend on $\epsilon$, while
the right-hand side tends to $-\infty$ as
$\epsilon\downarrow0$. Thus, we obtain that
\[
\limsup_{n\to\infty}\frac1{n^{p/2}}
\log\mathbb P\bigl(|D^B_{n,p}-p^{1/p}\sqrt{T_n}|>\delta\bigr)
=-\infty.
\]

For the boundary distance, the proof follows the same idea. Setting both radial factors equal to $1$ in the distance estimate, we obtain that
\begin{align*}
&\mathbb P\bigl(
|D^\partial_{n,p}-p^{1/p}\sqrt{T_n}|>\delta\bigr)\\
&\qquad\leq
2\mathbb P\bigl(|S_{4,n}^{-1/p}-p^{1/p}|>\epsilon\bigr)
+2\mathbb P\left(S_{1,n}>\frac{\delta^2}{4\epsilon^2}\right).
\end{align*}
Using~\eqref{EqNormalizerSuperexponential}
and~\eqref{EqSquareMeanTail}, and again taking the limit
superior as $n\to\infty$ before letting
$\epsilon\downarrow0$, we obtain exponential equivalence
for the boundary as well.

\medskip

We have proved that both distances are exponentially
equivalent to $p^{1/p}\sqrt{T_n}$ at speed $n^{p/2}$.
Since $p^{1/p}\sqrt{T_n}$ satisfies an LDP with this
speed and good rate function $I_p$, the exponential
equivalence lemma (Lemma~\ref*{LemExponentialEquivalence}) implies that $D^B_{n,p}$ and $D^\partial_{n,p}$ satisfy the same LDP with the same speed and rate function.
\end{proof}


\begin{thebibliography}{99} 


\bibitem{DF}{\sc P. Diaconis and D. Freedman.} \textit{A dozen de Finetti-style results in search of a theory.} Annales de l'I.H.P. Probabilit\'es et statistiques, Volume 23 (1987) no. S2, pp. 397-423.

\bibitem{DZ}{\sc A. Dembo and O. Zeitouni.} \textit{Large Deviations. Techniques and Applications.} volume 38 of Stochastic Modelling and Applied Probability, Berlin, 2010. Corrected reprint of the second (1998) edition.

\bibitem{GKR}{\sc N. Gantert, S.S. Kim, and K. Ramanan.} \textit{Large deviations for random projections of $l_p$ balls.}, The Annals of Probability, vol. 45, no. 6B, 2017, pp. 4419-76.

\bibitem{H}{\sc J.M. Hammersley} \textit{The Distribution of Distance in a Hypersphere.} The Annals of Mathematical Statistics, vol. 21, no. 3, 1950, pp. 447-52.

\bibitem{K}{\sc B. Klartag} \textit{A central limit theorem for convex sets.} Invent. Math., 168(1):91-131, 2007.

\bibitem{K2}{\sc B. Klartag} \textit{Power-law estimates for the central limit theorem for convex sets.}J. Funct. Anal., 245(1):284-310, 2007.

\bibitem{L}{\sc R.D. Lord} \textit{The Distribution of Distance in a Hypersphere.} The Annals of Mathematical Statistics, vol. 25, no. 4, 1954, pp. 794-98.

\bibitem{O}{\sc G. W. Oehlert} \textit{A Note on the Delta Method.} The American Statistician. 46 (1): 27-29, 1992.

\bibitem{PPZ}{\sc G. Paouris, P. Pivovarov, and J. Zinn.} \textit{A central limit theorem for projections of the cube.} Probab. Theory Related Fields, 159(3-4):701-719, 2014.

\bibitem{P}{\sc J. Prochno} \textit{The large and moderate deviations approach in geometric functional analysis.} High Dimensional Probability X, Progress in Probability, The Bedlewo Volume, pp. 171-245, 2026

\bibitem{PTT}{\sc J. Prochno, C. Th\"ale, P. Tuchel.} \textit{Limit Theorems for the Volume of Random Projections and Sections of $l_p^N$-balls.} 10.48550/arXiv.2412.16054.


\bibitem{S}{\sc H. Scheff\'e} \textit{ A Useful Convergence Theorem for Probability Distributions.} The Annals of Mathematical Statistics 18 (3), 434-438, 1947.

\bibitem{SZ}{\sc G. Schechtman, J. Zinn.} \textit{ On the volume of the intersection of two $L_p^n$ balls.} Proc. Amer. Math. Soc. 110(1):217-224, 1990.

\bibitem{St}{\sc A. J. Stam} \textit{Limit theorems for uniform distributions on spheres in high-dimensional Euclidean spaces.} J. Appl. Probab., 19(1):221-228, 1982.










\bibitem{GRR}{\sc N. Gantert, K. Ramanan, and F. Rembart.}
\textit{Large deviations for weighted sums of stretched exponential random variables.}
Electronic Communications in Probability 19 (2014), paper no. 41, 1--14.
\url{https://doi.org/10.1214/ECP.v19-3266}.

\bibitem{KPT}{\sc Z. Kabluchko, J. Prochno, and C. Th\"ale.}
\textit{High-dimensional limit theorems for random vectors in $\ell_p^n$-balls.}
Communications in Contemporary Mathematics 21 (2019), no. 1, article 1750092, 30 pp.
\url{https://doi.org/10.1142/S0219199717500924}.


\end{thebibliography}
\end{document}